\documentclass[10pt,twoside,english]{amsart}
\normalsize
\advance\oddsidemargin by -1.0cm
\advance\evensidemargin by -1.0cm
\advance\topmargin by -1.0cm 
\usepackage{amssymb}
\usepackage{babel}
\usepackage{amsmath}
\usepackage{amscd}   
\usepackage{epsfig}  
\usepackage{rotating}
\usepackage{psfrag}
\let\mathg\mathfrak
\theoremstyle{plain}
\newtheorem{cor}{Corollary}[section]
\newtheorem{lem}{Lemma}[section]
\newtheorem{thm}{Theorem}[section]
\newtheorem{prop}{Proposition}[section]

\theoremstyle{definition}
\newtheorem{exa}{Example}[section]
\newtheorem{NB}{Remark}[section]

\newcommand{\bdm}{\begin{displaymath}}
\newcommand{\edm}{\end{displaymath}}
\newcommand{\be}{\begin{equation}}
\newcommand{\ee}{\end{equation}}
\newcommand{\ba}[1]{\begin{array}{#1}}
\newcommand{\ea}{\end{array}}
\newcommand{\bea}[1][]{\begin{eqnarray#1}}
\newcommand{\eea}[1][]{\end{eqnarray#1}}
\newcommand{\btab}{\begin{tabular}}
\newcommand{\etab}{\end{tabular}}

\newcommand{\x}{\times}

\newcommand{\ox}{\otimes}

\newcommand{\ra}{\rightarrow}
\newcommand{\lra}{\longrightarrow}

\newcommand{\lmapsto}{\longmapsto}

\newcommand{\Id}{\ensuremath{\mathrm{Id}}}
\newcommand{\tr}{\ensuremath{\mathrm{tr}}}

\newcommand{\C}{\ensuremath{\mathbb{C}}}
\newcommand{\R}{\ensuremath{\mathbb{R}}}

\newcommand{\Z}{\ensuremath{\mathbb{Z}}}

\renewcommand{\P}{\ensuremath{\mathbb{P}}}

\newcommand{\HP}{\ensuremath{\mathbb{HP}}}
\newcommand{\CP}{\ensuremath{\mathbb{CP}}}

\newcommand{\eps}{\ensuremath{\varepsilon}} 
\newcommand{\vphi}{\ensuremath{\varphi}}    
\newcommand{\vrho}{\ensuremath{\varrho}}

\newcommand{\End}{\ensuremath{\mathrm{End}}}

\newcommand{\Ric}{\ensuremath{\mathrm{Ric}}}
\newcommand{\Scal}{\ensuremath{\mathrm{Scal}}}

\newcommand{\Ad}{\ensuremath{\mathrm{Ad}\,}}

\newcommand{\diag}{\ensuremath{\mathrm{diag}}}

\newcommand{\st}{\ensuremath{\mathrm{st}}}
\newcommand{\ir}{\ensuremath{\mathrm{ir}}}

\newcommand{\slin}{\ensuremath{\mathg{sl}}}
\newcommand{\SL}{\ensuremath{\mathrm{SL}}}

\newcommand{\su}{\ensuremath{\mathg{su}}}
\newcommand{\SU}{\ensuremath{\mathrm{SU}}}

\newcommand{\Sympl}{\ensuremath{\mathrm{Sp}}}

\newcommand{\so}{\ensuremath{\mathg{so}}}
\newcommand{\SO}{\ensuremath{\mathrm{SO}}}
\newcommand{\Spin}{\ensuremath{\mathrm{Spin}}}

\newcommand{\g}{\ensuremath{\mathfrak{g}}}
\newcommand{\h}{\ensuremath{\mathfrak{h}}}
\newcommand{\m}{\ensuremath{\mathfrak{m}}}
\newcommand{\n}{\ensuremath{\mathfrak{n}}}

\newcommand{\nms}{\!\!}%
\begin{document}
\def\haken{\mathbin{\hbox to 6pt{%
                 \vrule height0.4pt width5pt depth0pt
                 \kern-.4pt
                 \vrule height6pt width0.4pt depth0pt\hss}}}
    \let \hook\intprod
\setcounter{equation}{0}
%
%
\thispagestyle{empty}
%
\date{\today}
\title[$\SO(3)$-manifolds]{On the topology and the geometry of $\SO(3)$-manifolds}
%
%
%
\author{Ilka Agricola}
\author{Julia Becker-Bender}
\author{thomas Friedrich}
\address{\hspace{-5mm} 
Ilka Agricola, Julia Becker-Bender \newline
Fachbereich Mathematik und Informatik \newline
Philipps-Universit\"at Marburg\newline
Hans-Meerwein-Strasse \newline
D-35032 Marburg, Germany\newline
{\normalfont\ttfamily agricola@mathematik.uni-marburg.de}\newline
{\normalfont\ttfamily beckbend@mathematik.uni-marburg.de}}
\address{\hspace{-5mm} 
Thomas Friedrich\newline
Institut f\"ur Mathematik \newline
Humboldt-Universit\"at zu Berlin\newline
Sitz: WBC Adlershof\newline
D-10099 Berlin, Germany\newline
{\normalfont\ttfamily friedric@mathematik.hu-berlin.de}}
%
%
\thanks{Supported by the Junior Research Group "Special Geometries in Mathematical Physics"
of the Volkswagen\textbf{Stiftung}.}
\subjclass[2000]{Primary 53 C 25; Secondary 81 T 30}
\begin{abstract}
Consider the nonstandard embedding of $\SO(3)$ into $\SO(5)$ given by
the $5$-dimensional irreducible representation of $\SO(3)$, henceforth
called $\SO(3)_\ir$. In this note, we study the 
topology and the differential geometry
of $5$-dimensional Riemannian manifolds carrying such an $\SO(3)_\ir$ 
structure, 
i.\,e.~with a reduction of the frame bundle to $\SO(3)_\ir$. 
\end{abstract}
\maketitle
\pagestyle{headings}
%
%
%
\section{Introduction}\noindent
%
We consider the nonstandard embedding of $\SO(3)$ into $\SO(5)$ given by
the $5$-dimensional irreducible representation of $\SO(3)$, henceforth
called $\SO(3)_\ir$.
In this note, we investigate the topology and the differential geometry
of $5$-dimensional Riemannian manifolds carrying such an $\SO(3)_\ir$ 
structure, 
i.\,e.~with a reduction of the frame bundle to $\SO(3)_\ir$. 
These spaces are the 
non-integrable analogues of the symmetric space $\SU(3)/\SO(3)$
and its non-compact dual $\SL(3,\R)/\SO(3)$. 
While the general frame work for the investigation of such structures 
was outlined in \cite{Fri2}, first concrete results were 
obtained by M.~Bobienski and P.~Nurowski (general theory, \cite{Bobienski&N07})
as well as S.\,G.~Chiossi and A.~Fino ($\SO(3)_\ir$ structures on 
$5$-dimensional Lie groups, \cite{Chiossi&F}).

In the first part of the paper, we describe the topological
properties of the two different types of $\SO(3)$ structures.
While classical results by  E.~Thomas and M.~Atiyah are available
for the standard diagonal embedding of $\SO(3)_\st \ra \SO(5)$, the
case of $\SO(3)_\ir$  is first investigated in this paper.
We show that the symmetric space $\SU(3)/SO(3)$ admits a
 $\SO(3)_\ir$ structure, but 
no $\SO(3)_\st$ structure. We prove necessary
relations for the characteristic classes of a $5$-manifold
with a topological $\SO(3)_\ir$  structure: its first
Pontrjagin class $p_1(M)$ has to be divisible by five,
the Stiefel-Whitney classes $w_1(M),\, w_4(M),\, w_5(M)$ vanish etc.
Moreover, a simply-connected
$\SO(3)_\ir$-manifold that is spin is automatically parallelizable.
We construct explicit examples of $S^1$-fibrations over a 
$4$-dimensional base that admit a topological $\SO(3)_\ir$ structure.

In the second part, the differential geometry of some homogeneous
examples is studied in detail. We will focus on a `twisted' Stiefel manifold
$V^\ir_{2,4}=\SO(3)\x \SO(3)/ \SO(2)_\ir$, its non compact partner
$\tilde{V}^\ir_{2,4}=\SO(2,1)\x \SO(3)/ \SO(2)_\ir$ and the
space $W^\ir=\R\x (\SL(2,\R)\ltimes\R^2)/ \SO(2)_\ir$.
On each of these, a family of metrics depending on three
deformation parameters $\alpha,\beta,\gamma$ is considered; in the
case $W^\ir$ it is in addition necessary to consider a family
of possibble embeddings of $\SO(2)_\ir$ into $\R\x (\SL(2,\R)\ltimes\R^2)$,
as the ones leading to $\SO(3)_\ir$ structures are far from trivial.
The standard Stiefel manifold is known to admit an Einstein-Sasaki metric
that was crucial for the understanding of Riemannian Killing spinors.
In contrast, we show that the twisted Stiefel manifold admits
a nearly integrable $\SO(3)_\ir$ structure with parallel torsion
(it can even be naturally reductive for some parameters of the metric),
a compatible Sasaki structure whose contact connection coincides
with the $\SO(3)_\ir$ connection, but none of
the Sasaki structures is Einstein (however, an Einstein metric is shown to 
exist). All in all, the twisted Stiefel manifold is an example
of a rather well-behaved $\SO(3)_\ir$-manifold.
The manifold $W^\ir$ carries an $\SO(3)_\ir$ structure that disproves several
conjectures on $\SO(3)_\ir$-manifolds that one might be tempted to conclude
from the previous example. It carries a  $\SO(3)_\ir$ structure only for two
possible embeddings of $\SO(2)_\ir$ that depend on the parameters  
$\alpha,\beta,\gamma$ of the metric. The torsion of the $\SO(3)_\ir$
connection turns out to be non parallel, the space is never Einstein and
never naturally reductive, and there does not exist a compatible 
contact structure whose contact connection would coincide with the
$\SO(3)_\ir$ connection. In particular, this shows that a $\SO(3)_\ir$ 
structure is conceptionally really different from
a contact structure; it thus defines a new type of geometry
on $5$-manifolds.

\section{General remarks on $\SO(3)$ structures}\noindent
%

The Lie group $\SO(3)$ admits two inequivalent embeddings into $\SO(5)$.
The standard embedding is as upper diagonal matrices, 
$\SO(3)_\st\subset \SO(5),\ A\mapsto \diag(A,1,1)$, while the second
embedding corresponds to the unique faithful irreducible $5$-dimensional 
representation of  $\SO(3)$ and will henceforth be denoted 
by $\SO(3)_\ir\subset\SO(5)$. A realization of this representation which is
particular 
adapted to the spirit of this note is by conjugation on symmetric trace free 
endomorphisms of $\R^3$, denoted by  $S^2_0(\R^3)$, 
\bdm
\vrho(h)X\ :=\ hXh^{-1} \text{ for }h\in\SO(3),\ X\in S^2_0(\R^3)\ \cong\ \R^5.
\edm
If we choose the following basis for $S^2_0(\R^3)$, 
\bdm
X\ =\ \sum_{i=1}^5 x_i e_i\ =\ 
\begin{bmatrix} \frac{x_1}{\sqrt{3}} -x_5 & x_4 & x_2\\
x_4 & \frac{x_1}{\sqrt{3}} +x_5&x_3\\ 
x_2 & x_3 & -2\frac{x_1}{\sqrt{3}}\end{bmatrix},
\edm
and denote by $E_{ij}$ the endomorphism sending $e_i$ to $e_j$, $e_j$ to $-e_i$
and everything else to zero, the Lie algebras of the two embeddings 
above are spanned by the  bases
\begin{gather*}
\so(3)_\st\ =\ \langle s_1 :=E_{23},\, s_2:=E_{31},\, s_3:=E_{12}
\rangle,\quad
\so(3)_\ir\ =\ \langle X_1,\, X_2,\, X_3\rangle,\\
X_1=\vrho (s_1)= \sqrt{3}\,E_{13}+E_{42}+E_{53},\,
X_2=\vrho (s_2) = \sqrt{3}\,E_{21}+E_{34}+E_{52},\, X_3=\vrho (s_3)=E_{23}
+2 E_{45} .
\end{gather*}
By definition, a $\SO(3)_\st$ resp.~$\SO(3)_\ir$ structure on a $5$-manifold
is a reduction of its frame bundle to a subgroup $\SO(3)\subset \SO(5)$
isomorphic to  $\SO(3)_\st$ resp.~$\SO(3)_\ir$. The first example of a 
manifold with a $\SO(3)_\ir$ structure is the Riemannian symmetric space 
$\SU(3)/SO(3)$ with its natural Sasaki-Einstein metric, see 
\cite{Bobienski&N07} for a detailed description.

One crucial observation of \cite{Bobienski&N07} is that $\so(3)_\ir$ may
be characterized as being the isotropy group of of a symmetric 
$(3,0)$-tensor $\Upsilon$ on $\R^5$. Basically, this symmetric tensor
is one of the coefficients of the characteristic polynomial of
$X\in S^2_0(\R^3)$, more precisely,
\bdm
\det (X-\lambda\,\Id)\ =\ -\lambda^3+g(X,X)\lambda -\frac{2\sqrt{3}}{9}\,
\Upsilon(X,X,X).
\edm
The coordinates are chosen in such a way that the bilinear form
$g$ takes the simple expression $g(X,X)=\sum_{1}^5 x_1^2$, while $\Upsilon$
is the homogeneous polynomial
\bdm
\Upsilon(X,X,X)\ =\ x_1^3+\frac{3}{2}x_1(x_2^2+x_3^2-2x_4^2-2x_5^2)
+\frac{3\sqrt{3}}{2}(x_2^2-x_3^2)x_5 -3\sqrt{3}\, x_2 x_3 x_4.
\edm
In fact, a $\SO(3)_\ir$ structure on a  $5$-dimensional Riemannian manifold 
$(M^5,g)$ can equally be characterised as being a rank $3$ tensor field
$\Upsilon$ for which the associated linear map
$TM\ra \End(TM),\ v\mapsto \Upsilon_v$ defined by 
$(\Upsilon_v)_{ij}=\Upsilon_{ijk}v_k$ satisfies
\begin{enumerate}
\item it is totally symmetric: $g(u,\Upsilon_v  w)= g(w,\Upsilon_v u)=
g(u,\Upsilon_w v)$,
\item it is trace-free: $\tr\Upsilon_v=0$,
\item it reconstructs the metric: $\Upsilon^2_v v =g(v,v)v$.
\end{enumerate}
Recall that for a $G$ structure, a metric connection $\nabla$ 
is called a \emph{characteristic connection} if it is a $G$ connection 
whose torsion is totally antisymmetric \cite{Fri2}.  
\begin{thm}[{\cite{Bobienski&N07}}]
A $\SO(3)_\ir$ structure $(M,g,\Upsilon)$ can only admit a characteristic 
connection if it is \emph{nearly integrable}, i.\,e.~if the 
tensor $\Upsilon$ satisfies $(\nabla^g_v\Upsilon)(v,v,v)=$ for all
vector fields $v$. In this case, the torsion of the characteristic
connection is of algebraic type $\Lambda^3(\R^5)\cong\Lambda^2(\R^5)\cong
\so(5)=\so(3)_\ir \oplus\n$.
\end{thm}
The analogy to the definition of a nearly K\"ahler manifold is
evident. However, contrary to nearly K\"ahler manifold
(see \cite{Kirichenko77}, \cite{Alexandrov&F&S04}), the torsion
$T$ of the characteristic connection of a nearly integrable 
$\SO(3)_\ir$ structure is not always parallel. Examples will be discussed in
Section 4.
%
\section{Topological existence conditions}
%
Necessary and sufficient conditions for the existence of a 
$\SO(3)_\st$ structure on an oriented $5$-manifold were investigated in
the late 1960ies. In fact, the existence of such a structure is
equivalent to the existence of two global linearly independent vector fields.
Recalling that the Kervaire semi-characteristic is defined by
\bdm
k(M^5)\ :=\ \sum_{i=0}^2\dim_\R H^{2i}(M^5;\R) \quad \bmod{2},
\edm
one has the following classical result:
\begin{thm}[{\cite{Thomas68}, \cite{Atiyah70}}]
A $5$-dimensional compact oriented manifold admits  
two global linearly independent vector fields if and only if
\bdm
w_4(M^5)\ =\ 0 \ \text{ and }\ k(M^5)\ =\ 0.
\edm
\end{thm}
There is a second semi-characteristic,
\bdm
\hat{\chi}_2(M^5) \ := \ \sum_{i=0}^{2} \mathrm{dim}_{\Z_2}H_{i}(M^5 ;
\Z_2) \quad \bmod{2}  .
\edm
The Lusztig-Milnor-Peterson formula \cite{LusztigMP69} establishes the link
between these two semi-characteristics,
\bdm
k(M^5)\, - \, \hat{\chi}_2(M^5) =  w_2(M^5) \cup w_3(M^5) \, .
\edm
In particular, if $M^5$ is spin or $w_3(M^5) = 0$, then $k(M^5)=  
\hat{\chi}_2(M^5)$.\\

In \cite{Bobienski&N07} and \cite{Bobienski06}, it was claimed that
the existence of a $\SO(3)_\ir$ structure is equivalent to
the existence of a $\SO(3)_\st$ structure and the divisibility of the
first integral Pontrjagin class by five. However, the symmetric space 
$\SU(3)/\SO(3)$ is known to have a $\SO(3)_\ir$ structure, but it does 
not have a $\SO(3)_\st$ structure.
\begin{exa}
The symmetric space $M^5:= \SU(3)/\SO(3)$ admits a $\SO(3)_\ir$ structure, but 
no $\SO(3)_\st$ structure.
\end{exa}
\begin{proof}
Let us start by computing the isotropy representation of $M^5$. 
We choose as a basis of $\so(3)$ the elements $a_1 :=E_{23},\, a_2:=E_{31},
\, a_3:=E_{12}$ and complete it to a basis of $\su(3)$ by choosing
\begin{gather*}
b_1\ =\ i\,\begin{bmatrix}  0&1&0\\ 1&0&0\\ 0&0&0\end{bmatrix},\quad
b_2\ =\ i\,\begin{bmatrix}  0&0&1\\ 0&0&0\\ 1&0&0\end{bmatrix},\quad
b_3\ =\ i\,\begin{bmatrix}  0&0&0\\ 0&0&1\\ 0&1&0\end{bmatrix},\\
b_4\ =\ i\,\begin{bmatrix}  1&0&0\\ 0&-1&0\\ 0&0&0\end{bmatrix},\quad
b_5\ =\ \frac{i}{\sqrt{3}}\,\begin{bmatrix}  1&0&0\\ 0&1&0\\ 0&0&-2
\end{bmatrix}.
\end{gather*}
In this basis, the isotropy representation $\lambda:\so(3)\ra\so(5)$
is given by
\bdm
\lambda(a_1)\ =\ E_{12}+E_{34}-\sqrt{3}\,E_{35},\ \ 
\lambda(a_2)\ =\ -E_{13}+E_{24}+\sqrt{3}\,E_{25},\ \ 
\lambda(a_3)\ =\ -2\, E_{14}+E_{23}.
\edm
This representation is irreducible, hence isomorphic to the
$5$-dimensional irreducible representation of $\so(3)$.
$M^5$ cannot
admit a $\SO(3)_\st$ structure as claimed. Indeed, $M^5$ is a rational
homology sphere and a computation of the
$\Z_2$-cohomology yields the following result:
\bdm
H^1(M^5 ; \Z_2) \ = \ H^4(M^5 ; \Z_2) \ = \ 0 \, , \quad
H^2(M^5 ; \Z_2) \ = \ H^3(M^5 ; \Z_2) \ = \ Z_2 \ .
\edm
Consequently, we obtain $k(M^5)= 1$ and $\hat{\chi}_2(M^5)= 0$. Moreover, $M^5$ does not admit any $spin^{\C}$ structure (see
\cite{Dirac-Buch}, page $50$).
\end{proof}
Our description of the irreducible representation of $\SO(3)$ on $S^2_0(\R^3)$
implies the following characterization:
\begin{lem}\label{lem:endom-bundle}
A $5$-manifold $M^5$ admits a $\SO(3)_\ir$ structure if and only if there
exists a $3$-dimensional oriented vector bundle
$E^3\ra M^5$ such that the tangent bundle $TM^5$ is isomorphic to the
bundle of symmetric trace free endomorphisms of $E^3$, 
$TM^5\cong S^2_0(E^3)$.
\end{lem}
This characterization allows to formulate some  necessary topological
conditions for the existence of $\SO(3)_\ir$ structures.
\begin{thm}\label{thm.char-classes}
Suppose that $M^5$ admits a $\SO(3)_\ir$ structure and that $E^3$ is
a vector bundle over $M^5$ as just described. Then the following
relations hold:
\begin{enumerate}
\item $p_1(M^5)=5\, p_1(E^3)  \in H^4(M^5;\Z)$. \\
In particular, the first
  Pontrjagin class $p_1(M^5)$ is divisible by five.
\item $w_1(M^5)=0,\ w_4(M^5)=0,\ w_5(M^5)=0$,
\item $w_2(M^5)=w_2(E^3),\ w_3(M^5)=w_3(E^3)$.
\end{enumerate}
\end{thm}
\begin{proof}
We use the Borel-Hirzebruch formalism and consider, for a connected and
compact Lie group $G$ and its maximal torus $T \subset G$, the 
maps $H^*(BT)^{W} = H^*(BG) \rightarrow H^*(M^5)$ induced by the classifying
maps of the bundles $T(M^5)$ and
$E^3$. Denote by $\omega_1 , \omega_2$ the weights of $\SO(5)$ and by $\omega$
the weight of $\SO(3)$. Then we have
\bdm
1 \, + \, p_1(TM^5) \ = \ 1 \, + \, (\omega_1^2 \ + \ \omega_2^2) \, , \quad
1 \, + \, p_1(E^3) \ = \ 1 \, + \, \omega^2 \ .
\edm
The inclusion $\SO(3)_\ir \subset \SO(5)$ induces the map $\omega \rightarrow
(\omega_1 , 2 \, \omega_2)$ and we obtain
\bdm
p_1(M^5) \ = \ \omega_1^2 \, + \, 
\omega_2^2 \ = \ \omega^2 \, + \, (2 \, \omega)^2
\ = \ 5 \, \omega^ 2 \ = \ 5 \ p_1(E^3) \ .
\edm
If the Stiefel-Whitney classes of $E^3$ are given by the elementary symmetric
functions
\bdm
w(E^3) \ = \ (1 \, + \, x_1)(1 \, + \, x_2)(1 \, + \, x_3) \ ,
\edm
then the classes of $S_0^2(E^3)$ are computed by
\bdm
w(S_0(E^3)) \ = \ w(S(E^3)) \ = \ \prod_{1 \leq i \leq j \leq 3} (1 \, + \,
x_i \, + \, x_j) \ .
\edm
The bundle $E^3$ is oriented, $x_1 + x_2 + x_3 = w_1(E^3) = 0$. A direct
computation mod $2$ yields now the results
\bdm
w_1(S(E^3)) \ = \ w_4(S(E^3)) \ = \ w_5(S(E^3)) \ = \ 0 \ ,
\edm
and
\begin{eqnarray*}
w_2(S(E^3)) &=& (x_1 \, + \, x_2 \, + \, x_3)^2 \, + \, x_1x_2 \, + \,
x_1x_3 \, + \, x_2x_3 \ = \ w_2(E^3) \ , \\
w_3(S(E^3)) &=&  (x_1 \, + \, x_2 \, + \, x_3)( x_1x_2 \, + \,
x_1x_3 \, + \, x_2x_3) \, + \, x_1x_2x_3 \ = \ w_3(E^3) \ .
\end{eqnarray*}
\end{proof}
\begin{NB}
The construction and classification of oriented $3$-dimensional vector bundles
over compact, oriented $5$-manifolds in terms of topological data is difficult. But assume that we have such a bundle
$E^3$ over $M^5$ and $w_2(E^3) = w_2(M^5) \, , \, 5 \, p_1(E^3) = p_1(M^5)$ 
holds. Moreover, assume that $H^4(M^5; \Z) = H_1(M^5 ; \Z)$
has no $2$-torsion. Then $w_2(S^2_0(E^3)) = w_2(T(M^5)), \, p_1(S^2_0(E^3))
= p_1(T(M^5))$ and the real vector bundles $S^2_0(E^3)$ and $T(M^5)$ are
stable equivalent (see \cite{DW59} and \cite{Thomas68}).
\end{NB}
Wu's formulas linking the Stiefel-Whitney classes of an oriented $5$-manifold
$M^5$ read as
\begin{eqnarray*}
w_3(M^5) &=& Sq^1(w_2(M^5)) \, , \quad w_4(M^5) \ = \ w_2(M^5) \, \cup \, w_2(M^5)
\, ,  \\
w_2(M^5) \, \cup \, w_3(M^5) &=& Sq^1(w_2(M^5) \, \cup \, w_2(M^5)) \, + \
Sq^2(w_3(M^5)) \ . 
\end{eqnarray*}
In particular, we obtain
\begin{cor}
If $M^5$ admits a $\SO(3)_\ir$ or a $\SO(3)_\st$ structure, then
\bdm
w_2(M^5) \, \cup \, w_2(M^5) \ = \ 0 \, , \ w_3(M^5) \ = \ Sq^1(w_2(M^5)) \, ,
\quad
w_2(M^5) \, \cup \, w_3(M^5) \ = \ Sq^2(w_3(M^5))
\edm
holds.
\end{cor}

\begin{exa}
The real projective space $\R\P^5$ cannot have either kind of a
 $\SO(3)$ structure, for in both cases the vanishing of $w_4(\R\P^5)$
would be a necessary condition. Indeed, its Stiefel-Whitney class
$w_4(\R\P^5)\neq 0$ is non-trivial.
\end{exa}
The necessary conditions expressed via the Pontrjagin class as well as the
Stiefel-Whitney classes do not imply the existence of a $\SO(3)_\ir$
structure. In fact, there is a further obstruction in $H^5(M^5;\Z_2)$,
probably the vanishing of $\hat{\chi}_2(M^5)$. Here we prove only a weaker
statement.
\begin{thm}\label{thm:so3-parallel}
Let $M^5$ be a compact, simply-connected spin manifold admitting a
$\SO(3)_\ir$ or a $\SO(3)_\st$ structure. Then $M^5$ is parallelizable. In particular, the
sphere $S^5$ does not admit a $\SO(3)_\ir$ nor a $\SO(3)_\st$ structure.  
\end{thm} 
\begin{NB}
The Theorem is well known for the standard embedding. Indeed, if $w_2(M^5) =0$
and the simply-connected $M^5$ admits a $\SO(3)_\st$ structure, then
$p_1(M^5) = 0$ as well as $k(M^5) = \hat{\chi}_2(M^5) = 0$. These conditions
imply that $M^5$ is parallelizable (see \cite{Thomas68}).
\end{NB}
In order to prove this theorem, we need the following
\begin{lem}\label{lem:ind-hom}
Let $i : \SO(3) \rightarrow \SO(5)$ be the standard or the irreducible embedding
of the group $\SO(3)$ into $\SO(5)$. Then the induced homomorphism
\bdm
i_* \ : \ \pi_4(\SO(3)) \ = \ \Z_2 \ \longrightarrow \ \pi_4(\SO(5))\ = \ \Z_2
\edm 
is trivial.
\end{lem}
\begin{proof}[Proof of Lemma $\ref{lem:ind-hom}$]
We remark that both homotopy groups are isomorphic to $\Z_2$,
\begin{gather*}
\pi_4(\SO(3)) \ = \ \pi_4(\Spin(3)) \ = \ \pi_4(S^3) \ = \ \Z_2, \\
\pi_4(\SO(5)) \ = \ \pi_4(\Spin(5)) \ = \ \pi_4(\Sympl(2)) \ = \ 
\pi_4(\Sympl(1)) \ = \ \pi_4(S^3) \ = \ \Z_2 \ .
\end{gather*}
The second line is a consequence of the classical isomorphism $\Spin(5) = 
\Sympl(2)$ and the fibration $\Sympl(2)/\Sympl(1) = S^7$. First consider the
standard embedding $i_\st : \SO(3) \rightarrow \SO(5)$. Then
$\SO(5)/i_\st(\SO(3)) = V_{5,2}$ is the Stiefel manifold and we obtain the exact
sequence
\bdm
\ldots \lra \, \pi_4(\SO(3)) \, = \, \Z_2 \, \lra \, 
\pi_4(\SO(5) \, = \, \Z_2 \, \lra \, \pi_4(V_{5,2}) \, \lra
\, \pi_3(\SO(3)) \, = \, \Z \, \lra \, \ldots . 
\edm
Since $\pi_4(V_{5,2}) = \Z_2$ (see \cite{Pae56}) we conclude that $\pi_4(\SO(5))
\rightarrow \pi_4(V_{5,2})$ is surjective, i.e. $(i_\st)_* : \pi_4(\SO(3))
\rightarrow \pi_4(\SO(5))$ is trivial. If  $i_\ir : 
\SO(3) \rightarrow \SO(5)$ is
the irreducible embedding then we denote by $X^7 = \SO(5)/i_\ir(\SO(3)) =
\Sympl(2)/i_\ir(\Sympl(1))$ the corresponding homogeneous space (the so called
Berger space). Its homotopy groups are known,
\bdm
\pi_1(X^7) \ = \ \pi_2(X^7) \ = \ 0 \ , \quad \pi_3(X^7) \ = \ \Z_{10} \ .
\edm
The exact sequence of the homotopy groups of that fibration yields that $(i_\ir)_* : \pi_3(\Sympl(1)) = \Z \rightarrow 
\pi_3(\Sympl(2)) = \Z$ is multiplication by $10$. Consequently, the
map $i_\ir : S^3 = \Sympl(1) \rightarrow \Sympl(2)$ represents ten times the
generator $[h] \in \pi_3(\Sympl(2))$, $[i_\ir] = 10 \cdot [h]$. A similar
argument proves that $[i_\st] = 2 \cdot [h]$ holds. Fix a map $g : S^3 =
\Sympl(1) \rightarrow \Sympl(1) = S^3$ of degree $5$. Then we obtain
\bdm
[i_\st \circ g] \ = \ 10 \cdot [h] \ = \ [i_\ir] \ ,
\edm
i.\,e.~the maps $i_\st \circ g , \, i_\ir : \Sympl(1) \rightarrow \Sympl(2)$
are homotopic. The induced map $(i_\ir)_* = (i_\st)_* \circ
g_*$ of the embedding $i_\ir$ is given by the induced maps of 
$i_\st$ and of $g$. Finally we see that $(i_\ir)_* : \pi_4(\SO(3)) 
\rightarrow \pi_4(SO(5))$ is again trivial.
\end{proof}
\begin{proof}[Proof of Theorem $\ref{thm:so3-parallel}$]
Suppose that $M^5$ admits a $\SO(3)_\ir$ structure. Then $E^3$ is a 
$3$-dimensional, oriented bundle with a spin structure. Its frame
bundle $P_{E^3}$ is given by a classifying map $M^5 \rightarrow B\Spin(3) =
\HP^{\infty}$. The manifold $M^5$ is $5$-dimensional and, consequently, the
classifying map is a map into $S^4$. Since
$H^4(M^5;\Z) = H_1(M^5;\Z) = 0$ there are at most two homotopy classes
of maps from $M^5$ into $S^4$ (see \cite[Ch.\,8, S.\,5, Thm.\,15]{Span66}).
The frame bundle  $P_{E^3}$ is trivial
over the $4$-skeleton of $M^5$ and the two bundles are given by their 
obstruction classes in $H^5(M^5; \pi_4(\SO(3))) = \pi_4(\SO(3)) = \Z_2$. 
However, the map $i_* : \pi_4(\SO(3)) \rightarrow \pi_4(\SO(5))$ is trivial. 
This implies that the bundles $E^3 \oplus \theta^2$ (in case of the standard 
embedding) or $S^2_0(E^3)$ (in case of the irreducible embedding) are 
trivial. Finally, $M^5$ is parallelizable.
\end{proof}
\begin{exa}
The connected sums $(2 l + 1) \# (S^2 \times S^3)$ are simply-connected, spin
and they admit a $\SO(3)_\st$ structure (the Kervaire semi-characteristic $k$
vanishes). Therefore they are parallelizable.
\end{exa}
The subgroup
$\SO(2)_\ir  :=  \big\{ ( A, A^2, 1 ) : A \in \SO(2) \big\}$ 
is contained in $\SO(3)_\ir \subset \SO(5)$. A $5$-manifold $M^5$ admits
a $\SO(2)_\ir$ structure (and, in particular, a $\SO(3)_\ir$ structure) if and
only if there exists a complex line bundle $E$ such that $T(M^5) = E \oplus
E^2 \oplus \theta^1$. Suppose that $M^5$ is a $S^1$-fibration over a
$4$-manifold $X^4$. Then the tangent bundle of $X^4$ should split into
$T(X^4) = E \oplus E^2$. Let us discuss the latter condition. In this way we
are able to construct whole families of $5$-manifolds admitting a topological
$\SO(3)_\ir$ structure.
\begin{prop}
Let $X^4$ be a smooth, compact,  oriented $4$-dimensional 
manifold. Then the following conditions are equivalent:
\begin{enumerate}
\item The tangent bundle splits into $T(X^4) = E \oplus E^2$.
\item There exists an element $c \in H^2(X^4 ; \Z)$ such that
\bdm
p_1(X^4) \ = \ 5 \, c^2 \, , \quad 
\chi(X^4) \ = \ 2 \, c^2 \, , \quad
\mbox{and} \quad c \ \equiv \ w_2(X^4) \quad
\mbox{mod} \ 2 \ .
\edm
\item There exists an element $c \in H^2(X^4 ; \Z)$ such that
\bdm
\chi(X^4) \ = \ 2 \, c^2 \, , 
 \quad c \ \equiv \ w_2(X^4) \quad
\mbox{mod} \ 2 \, \quad \mbox{and} \quad 6 \, 
\sigma(X^4) \ = \ 5 \, \chi(X^4). 
\edm
\end{enumerate}
\end{prop}
\begin{proof}
If $T(X^4) = E \oplus E^2$ then the first Chern class $c= c_1(E)$ of the line
bundle $E$ satisfies all the conditions. Conversely, suppose that there exists
an element $c \in H^2(X^4 ; \Z)$ with the described properties. Then
we consider the line bundle $E$ defined by the condition $c = c_1(E)$. The
Euler, the Pontrjagin and the Stiefel-Whitney classes of the
real bundles $T(X^4)$ and $E \oplus E^2$ coincide and $H^4(X^4; \Z)$ has no
$2$-torsion. Consequently, the $4$-dimensional real vector bundles $T(X^4)$
and $E \oplus E^2$ are isomorphic (see \cite{DW59}, \cite{Thomas68}).
\end{proof}
Compact spin manifolds $X^4$ with finite fundamental 
group and
admitting a decomposition $T(X^4) = E \oplus E^2$ do not exist. Indeed,
denote by $U$ the
intersection form of $S^2 \times S^2$ and by $\Gamma_8$ the non-trivial,
positive definite quadratic form of rank $8$. The intersection form of the universal covering $\tilde{X}^4$ is isomorphic to
$p \cdot U \oplus q \cdot \Gamma_8$ (see \cite{Serre70}, chapter 5, 
Theorem 5). But
\bdm
\sigma(\tilde{X}^4) \ = \ 8 \, q  , \quad \chi(\tilde{X}^4) \ = \ 2 \, p \, + \, 8 \, q \, +
\, 2 
\edm 
and $6 \, \sigma(\tilde{X}^4) = 5 \,
\chi(\tilde{X}^4)$ yields  $8 q = 10 p + 10 > 8 p$. Finally we obtain
$p < q$ and the manifold cannot be smooth 
(the $11/8$ conjecture, see \cite{Furuta01}).\\

Any indefinite and odd quadratic form over $\Z$ can be
realized as the intersection form of a smooth, compact and simply-connected
$4$-manifold $X^4$. Any such form is isomorphic to the sum of two trivial
forms (see \cite{Serre70},
chapter 5, Theorem 4), $H^2(X^4;\Z) = s \cdot \langle 1\rangle \oplus \, t 
\cdot \langle -1 \rangle$. Then
we obtain
\bdm
\chi(X^4) \ = \ 2 \, + \, s  +  t , \quad \sigma(X^4) \ = \ s  -  t 
\edm
and the condition $6 \, \sigma(X^4) = 5 \, \chi(X^4)$ implies $s - 11 \, t =
10$. Consider the generators $e_1 , \ldots , e_s$ and $f_1 , \ldots , f_t$ of
the quadratic form with $e_{\alpha}^2 = 1$ as well as $f_{\beta}^2 =
-1$, $\alpha = 1 , \ldots , s$ and $\beta = 1 , \ldots , t$. An 
admissible class $c \equiv w_2(X^4)$ is a linear combination with odd
coefficients,
\bdm
c \ = \ a_1 \, e_1  +  \ldots   + \, a_s \, e_s \, + \, b_1 \, f_1 \, + \,
\ldots \,  + \, b_t \, f_t 
\edm  
and the equation $ 2\, c^2 = \chi(X^4)$ becomes
$a_1^2 +  \ldots   +  a_s^2  -  b_1^2   -  \ldots   -  
b_t^2 \, = \, 6 \, + \, 6 \, t$.
The system
\bdm
s \, - \, 11 \, t \ = \ 10 \, , \quad a_1^2 \, + \, \ldots \,  + \, a_s^2 \, - \, b_1^2  \, - \, \ldots \,  -  \,
b_t^2 \ = \ 6 \, + \, 6 \, t \ . 
\edm
has solutions in odd number $a_{\alpha}, b_{\beta}$. A first solution is
$ s = 21, \, t = 1,  \, a_1 = \ldots = a_{21} = 1, \, b_1 = 3$ and 
the corresponding manifold is homeomorphic to $21 \, 
\CP^2 \# \bar{\CP}^2$. A second solution is
$ s = 43, \, t = 3,  \, a_1 = \ldots = a_{43} = 3, \, b_1 =b_2 = b_3 = 11$
with the manifold $43 \, \CP^2 \, \# \, 3 \, \bar{\CP}^2$.
 A third solution is
$ s = 197, \, t = 17,  \, a_1 = \ldots = a_{197} = 15, \, b_1 = \ldots 
= b_{17} = 51$
with the manifold $197 \, \CP^2 \, \# \, 17 \, \bar{\CP}^2$.
Any $S^1$-bundle $M^5$ over these spaces admits a $\SO(2)_\ir \subset \SO(3)_\ir$
structure. 
\begin{NB}
The Thom-Gysin sequence yields the relations
\bdm
\hat{\chi}_2(M^5) \ \equiv \ \dim_{\Z_2}H_2(X^4; \Z_2) , \quad
k(M^5) \ \equiv \  \dim_{\R}H_2(X^4; \R)  , \quad
\hat{\chi}_2(M^5) \ = \ k(M^5)
\edm
for any oriented $S^1$-bundle $\pi : M^5 \rightarrow X^4$ over a compact, 
simply-connected $4$-manifold $X^4$. Note that 
$w_3(M^5) = \pi^*\big(w_3(X^4)) = 0$ holds anyway. $M^5$ is a spin manifold
if and only if the Chern class $c^* \in H^2(X^4;\Z)$ of the fibration  
$\pi : M^5 \ra X^4$ represents the Stiefel-Whitney class of $X^4$, 
$c^* \equiv w_2(X^4)$ mod $2$. The Pontrjagin class $p_1(M^5)$ vanishes if and only if there exists
an element $x \in H^2(X^4; \Z)$ such that $p_1(X^4) = c^* \cup x$.
\end{NB}
%
\section{Homogeneous examples and their geometric properties}
%
%
Homogeneous manifolds with a $\SO(3)_\ir$ structure have been
classified using Cartan's method of integration by Bobienski and Nurowski,
see \cite{Bobienski06}. In this section, it is our goal to describe their 
geometric properties. 

Given that a non-discrete subgroup  $H\subset \SO(3)_\ir$ can only have 
dimension $1$ or $3$,
a homogeneous space $M^5=G/H$ can only be the quotient of a group $G$
of dimension $8$ or $6$. The case $\dim G=8$ is not so interesting,
as these are precisely the symmetric spaces $\SU(3)/SO(3)$, $\SL(3,\R)/\SO(3)$
and $\R^5$. We shall  therefore concentrate our attention on homogeneous
spaces $M^5=G^6/\SO(2)$, whith $\SO(2)\subset \SO(3)_\ir$. The case of
$5$-dimensional Lie groups with $\SO(3)_\ir$ structure has been discussed in
detail by Chiossi and Fino, see \cite{Chiossi&F}.
%
%
%
\subsection{The `twisted' Stiefel manifold $V^\ir_{2,4}=\SO(3)\x \SO(3)/ \SO(2)_\ir$}
%
We consider the inclusion
\bdm
H \ :=\ \SO(2)\ni A\lmapsto (A,A^2)\in \SO(3)\x\SO(3)\ =:\ G
\edm
and the corresponding homogeneous space $V_{2,4}^\ir :=\SO(3)\x \SO(3)/ \SO(2)$.
If we choose as standard basis of the Lie algebra $\so(3)$ the elements
$s_1, s_2$ and $s_3$ defined in Section 2 and as basis of 
$\g=\so(3)\oplus\so(3)$ the elements $a_i=(s_i,0), b_i=(0,s_i), i=1,2,3$,
the Lie algebra $\h$ of $\SO(2)$ is  given by
\bdm
\h\ =\ \R\cdot \tilde{e}_0 \text{ with }  \tilde{e}_0\, =\,  (s_3,2 s_3)\, 
=\, (a_3+2b_3).
\edm
We further define 
\bdm
\tilde{e}_1\,=\, b_3-2a_3, \quad \tilde{e}_2\,=\, a_1,\quad
\tilde{e}_3\,=\, a_2,\quad \tilde{e}_4\,=\, b_1,\quad \tilde{e}_5\,=\, b_2.
\edm
As a reductive complement $\m$ of $\h$, we may then choose 
\bdm
\m\,=\, \n\oplus \m_1\oplus\m_2,\quad
\n\,=\, \R \cdot \tilde{e}_1,\quad
\m_1\,=\, \langle  \tilde{e}_2,  \tilde{e}_3\rangle,\quad
\m_2\,=\, \langle  \tilde{e}_4,  \tilde{e}_5\rangle.
\edm
One checks that the isotropy representation $\lambda=\Ad\big|_H:\, \SO(2)\ra\SO(5)$
is given by $\lambda (A)\ =\ \diag (1,A,A^2)$
and has differential $d\lambda(\tilde{e}_0)=E_{23}+E_{45}$.
In particular, one sees that $\lambda(H)$ is indeed a subgroup of 
$\SO(3)_{\ir}\subset\SO(5)$, but not of $\SO(3)_{\st}\subset\SO(5)$.
Thus, the quotient $G/H$ has a $\SO(3)_\ir$ structure as claimed.
In order to make this property more transparent, we shall write
$\SO(2)_\ir$ for the chosen embedding of $\SO(2)$ inside $\SO(3)\x\SO(3)$
as well as for its image $\lambda(\SO(2)_\ir)\subset \SO(5)$.

In order to define a suitable family of metrics on $\m$ and thereby
of Riemannian metrics on $G/H$, we first note that not only $\m$ itself,
but each of the spaces $\n,\m_1,\m_2$ is $\lambda(H)$ invariant. Thus,
it makes sense to consider a metric that is a renormalization of
the Killing form on each of the factors $\n,\m_1$ and $\m_2$. Since
our initial basis consists of orthogonal vectors for the Killing form,
we define a $3$-parameter family of metrics by
\bdm
g_{\alpha\beta\gamma}\ =\ \diag (\alpha, \beta,\beta,\gamma,\gamma),\quad
\alpha,\beta,\gamma\, >\, 0.
\edm 
We set $e_0=\tilde{e}_0$ and renormalize our previous basis so that it
becomes an orthonormal basis for this new metric,
\bdm
e_1\, =\, \frac{\tilde{e}_1}{\sqrt{\alpha}},\quad
e_2\, =\, \frac{\tilde{e}_2}{\sqrt{\beta}},\quad
e_3\, =\, \frac{\tilde{e}_3}{\sqrt{\beta}},\quad
e_4\, =\, \frac{\tilde{e}_4}{\sqrt{\gamma}},\quad
e_5\, =\, \frac{\tilde{e}_5}{\sqrt{\gamma}}.
\edm
For later reference, we state all non-vanishing commutator relations:
\bea[*]
[e_0,e_2]\nms &=& \nms e_3,\quad [e_0,e_3]\,=\, -e_2,\quad [e_0,e_4]\,=\, 2e_5,\quad
[e_0,e_5]\,=\,-2 e_4 , \\
{[e_1,e_2]}\nms &=&\nms -\frac{2}{\sqrt{\alpha}} e_3,\quad
[e_1,e_3]\,=\, \frac{2}{\sqrt{\alpha}}e_2,\quad
[e_1,e_4]\,=\,\frac{1}{\sqrt{\alpha}} e_5,\quad
[e_1,e_5]\,=\, - \frac{1}{\sqrt{\alpha}}e_4,\\
{[e_2,e_3]}\nms & = &\nms \frac{1}{5\beta}(e_0-2\sqrt{\alpha} \, e_1),\quad
[e_4,e_5]\,=\, \frac{1}{5\gamma}(2e_0+\sqrt{\alpha}\, e_1).
\eea[*]
\begin{NB}
Before investigating its properties in more detail, let us compare
the homogeneous space $V^\ir_{2,4}=\SO(3)\x\SO(3)/SO(2)_{\ir}$ with the classical
Stiefel manifold $V^\st_{2,4} = \SO(4)/\SO(2)$. Because of $\SO(4)=S^3\x\SO(3)$,
the corresponding Lie algebra embedding $\so(2)\ra \so(3)\x\so(3)$
is just $h\mapsto (h,0)$ in the classical case, $h\mapsto (h,2h)$ in the
case that we are considering. Hence, we see that that  $V^\st_{2,4}$ carries
an $\SO(3)_\st$ structure, hence justifying the superscript. 
In \cite{Jensen75}, it was shown that the classical
Stiefel manifold $V^\st_{2,4}$ carries an Einstein metric. Later, this 
Einstein metric
was recognized to be Sasaki and the existence of two Riemannian Killing
spinors was established \cite{Friedrich80}. Connections with
antisymmetric torsion on $V^\st_{2,4}$ were investigated in \cite{Agricola03}.
All in all, this example turned out to be crucial for the understanding of
the relations between contact structures and the existence of Killing spinors.
\end{NB}
\begin{NB}
Let us display the $\SO(3)_\ir$ structure of $V^\ir_{2,4}$ in yet another
way, namely, as a bundle $E^3$ satisfying $S^2_0(E^3)\cong T V^\ir_{2,4}$
as used in Section 3. The frame bundle of the homogeneous space
$V^\ir_{2,4}$ is $\mathcal{R}= G\times_{\lambda(H)}\SO(5)$ and its
tangent bundle is $TV^\ir_{2,4} = G\times_{\lambda(H)} \m$; therefore, the
following vector bundle is well defined,
\bdm
E^3\  =\ G\x _{\lambda(H)} \so(3)_{\ir},
\edm
where the action of $H$ is by conjugation on the subspace 
$\so(3)_{\ir}\subset\so(5)$ as always. One then checks that, as $H$
representations, $S^2_0(\so(3)_{\ir})\cong \m$, hence showing
$S^2_0(E^3)\cong T V^\ir_{2,4}$ as claimed.
\end{NB}
\begin{thm}[Connection properties]\label{thm-connection}
The twisted Stiefel manifold $V^\ir_{2,4}=\SO(3)\x\SO(3)/SO(2)_{\ir}$ 
equipped with the
family of metrics $g_{\alpha\beta\gamma}$ has the following properties:
\begin{enumerate}
\item For parameters $\alpha, \beta,\gamma >0$ satisfying
$\alpha\beta+4\,\gamma\alpha-25\,\beta\gamma=0$, the $\SO(3)_\ir$ structure
is nearly integrable and the  torsion $T^{\alpha\beta\gamma}$ of its
characteristic connection $\nabla^{\alpha\beta\gamma}$  is, in a suitable
orthonormal basis, given by
\bdm
T^{\alpha\beta\gamma} \ =\ \frac{2\sqrt{\alpha}}{5\beta}\,e_1\wedge e_2\wedge
e_3 - \frac{\sqrt{\alpha}}{5\gamma}\, e_1\wedge e_4\wedge e_5.
\edm
Its holonomy is $\SO(2)_\ir\subset \SO(5)$ and its torsion is parallel,
$\nabla^{\alpha\beta\gamma} T^{\alpha\beta\gamma} =0$.  
\item The metric of the nearly integrable $\SO(3)_\ir$ structure is naturally 
reductive if and only if $\alpha=5\beta=5\gamma$.
\end{enumerate}
\end{thm}
\begin{proof}
By a Theorem of Wang \cite[X.2]{Kobayashi&N2}, invariant metric
connections $\nabla^{\alpha\beta\gamma}$ on $V^\ir_{2,4}=G/H$ are in bijective 
correspondence with linear maps
$\Lambda_\m: \, \m\ra\so(5)$ that are equivariant under the adjoint
representation, 
\bdm\tag{$*$}
\Lambda_\m (hXh^{-1})\ =\ \Ad(h)\Lambda_\m (X)\Ad(h)^{-1} \quad
\forall h\in H,\ X\in \m.
\edm
Since  $\Lambda_\m$ is basically the connection form,
$\nabla^{\alpha\beta\gamma}$ will be a $\SO_\ir$ connection if and only if
 $\Lambda_\m$ takes values in the structure group $\SO(3)_\ir$ of the reduction
of the frame bundle. We have $\lambda(e_0)=e_{23}+2\,E_{45}= X_3$ in the 
notation of Section $2$ and complete it to a basis  of 
$\so(3)_\ir\subset \so(5)$ by choosing as additional elements $X_1$ and
$X_2$.  Thus, $\Lambda_\m(e_i)$ is a priori for each $i=1,\ldots,4$
a linear combination of the elements $X_i,\ i=1,2,3$. However,
the equivariance condition $(*)$ further restrics the possible values
of $\Lambda_\m(e_i)$; one checks that the most general Ansatz
for a $\SO_\ir$ connection is ($a,b,c\in\R$)
\bdm
\Lambda_\m(e_1)\,=\, a X_3,\quad
\Lambda_\m(e_2)\,=\, b X_1 -c X_2,\quad
\Lambda_m(e_3)\,=\, c X_1+b X_2,\quad
\Lambda_m(e_4)\,=\,\Lambda_m(e_5)\, =\, 0.
\edm
For the possible torsion $T^{\alpha\beta\gamma}\in \Lambda^3(G/H)$, observe 
that the only $\lambda$-invariant $3$-forms are $e_1\wedge e_2\wedge e_3$
and   $e_1\wedge e_4\wedge e_4$, thus the torsion has to be of the form
($m,n\in\R$)
\bdm
T^{\alpha\beta\gamma}\ = \ m\, e_1\wedge e_2\wedge e_3 +
n\, e_1\wedge e_4\wedge e_5.
\edm
Since the torsion of the connection defined by $\Lambda_\m$ is given by
\cite[X.2.3]{Kobayashi&N2}
\bdm\tag{$*$}
T(X,Y)_o\ = \ \Lambda_\m(X)Y-\Lambda_\m(Y)X- [X,Y]_\m, \quad X,\,Y\in\m,
\edm
one concludes by a routine evaluation on all pairs of vectors $e_i\neq e_j$
that
\bdm
b\,=\,c\,=\, 0,\quad 
m\,=\, a+\frac{2}{\sqrt{\alpha}}\,=\, \frac{2\sqrt{\alpha}}{5\beta},
\quad
n\,=\, 2a - \frac{1}{\sqrt{\alpha}}\,=\, - \frac{\sqrt{\alpha}}{5\gamma}.
\edm
A $\SO(3)_\ir$ connection is thus obtained  if and only if
$\alpha\beta+4\alpha\gamma-25\beta\gamma=0$ and is then defined by
\bdm
\Lambda_\m(e_1)\,=\,  \left( \frac{2\sqrt{\alpha}}{5\beta} - 
\frac{2}{\sqrt{\alpha}} \right) X_3, \quad
T^{\alpha\beta\gamma}\,=\, \frac{2\sqrt{\alpha}}{5\beta}\,e_1\wedge e_2\wedge
e_3 - \frac{\sqrt{\alpha}}{5\gamma}\, e_1\wedge e_4\wedge e_5.
\edm
If one requires further that $T^{\alpha\beta\gamma}(X,Y,Z)=-g([X,Y]_\m,Z)$,
one obtains a naturally reductive space and a comparison with the
commutator relations yields the stronger condition
$\alpha=5\beta=5\gamma$. Indeed, under this condition $\Lambda_\m=0$,
and the characteristic connection coincides with the canonical connection.

We now show that the torsion $T^{\alpha\beta\gamma}$ is 
$\nabla^{\alpha\beta\gamma}$-parallel. Invariant tensor are parallel with
respect to the canonical connection defined by $\Lambda_\m^c=0$, hence
 $\nabla^{\alpha\beta\gamma}_{e_i}T=\Lambda_\m(e_i)T$. We first note that 
on the invariant vector $e_1$, trivially $\Lambda_\m(e_1)e_1=0$ holds,
and after the identification $\so(\m)\cong \Lambda^2(\m)$, the action of
the connection on $2$-forms $\omega$ is given 
\bdm\tag{$**$}
\nabla^{\alpha\beta\gamma}_{e_i}\omega\ =\ 
\Lambda_\m(e_i)\omega\ =\ \sum_{j=1}^5 (e_j\haken\Lambda_\m (e_i))\wedge
(e_j\haken\omega).
\edm
Thus, one checks that $\nabla^{\alpha\beta\gamma} (e_2\wedge e_3) = 
 \nabla^{\alpha\beta\gamma} (e_4\wedge e_5) = 0$ and, consequently, 
  $\nabla^{\alpha\beta\gamma}T^{\alpha\beta\gamma} =0$ holds.

Finally, it is worth computing the curvature of the characteristic connection. 
In gereral, it is given by \cite[X.2.3]{Kobayashi&N2}
\bdm
R(X,Y)_o\ =\  [\Lambda_\m (X), \Lambda_\m (Y)] - \Lambda_\m ([X,Y]_\m)
-\lambda([X,Y]_\h) .
\edm
In the case at hand, one obtains as the only non-vanishing terms
\bdm
R(e_2, e_3)\ =\ \frac{1}{5\beta}\left[\frac{4\alpha}{5\beta} - 5\right] X_3,
\quad
R(e_4,e_5)\ =\ -\, \frac{2\alpha}{25\, \beta\gamma} X_3.
\edm
Let $\m_0\subset \m$ be the space spanned by the non-vanishing curvature
transformations. The holonomy algebra of the connection is then
\bdm
\mathg{hol}(\nabla^{\alpha\beta\gamma})\ =\ 
\m_0+[\Lambda_\m(\m),\m_0]+ [\Lambda_\m(\m),[\Lambda_\m(\m),\m_0]]+\ldots
\ =\ \R\cdot X_3\ =\ \so(2)_\ir\subset\so(5). \qedhere
\edm
\end{proof}
\begin{NB}
In the positive quadrant $\{(\alpha,\beta,\gamma)\in\R^3\ :\ \alpha>0,\
\beta>0,\ \gamma>0\}$, the metrics $g_{\alpha\beta\gamma}$ defining a nearly
integrable $\SO(3)_\ir$ structure form a ruled surfaces of rays through 
the origin. The naturally reductive metrics are exactly one ray (line) on this
ruled surface.
\end{NB}
\begin{thm}[Riemannian Curvature properties]\label{riem-curv-prop}
%
\begin{enumerate}
\item[]
\item The Riemannian Ricci tensor of the metric $g_{\alpha\beta\gamma}$ is given by
$\displaystyle 
\Ric(e_1,e_1)\,=\, \frac{2\alpha}{25\beta^2}+\frac{\alpha}{50\gamma^2}$,
\bdm
\Ric(e_2,e_2)\,=\, \Ric(e_3,e_3)\,=\, \frac{1}{\beta} - 
\frac{2\alpha}{25\beta^2},\quad
\Ric(e_4,e_4)\,=\, \Ric(e_5,e_5)\,=\, \frac{1}{\gamma} - 
\frac{\alpha}{50\gamma^2}.
\edm
\item The metric is Einstein (but not naturally reductive) for $\alpha=1$
and 
\bdm
\beta\, =\, \frac{18}{25(10^{2/3} -2\cdot
  10^{1/3}+4)}\ \cong\ 0.166177,\ 
\gamma\, =\, \frac{3(9+\sqrt{21+12\cdot 10^{1/3}+3\cdot 10^{2/3} })}{25(20-
4\cdot 10^{1/3} - 10^{2/3})}\, \cong\, 0.299009.
\edm
The Riemannian scalar curvature is $\Scal^g=\left[\frac{5}{3}\right]^3
(100-2\cdot 10^{1/3} - 10^{2/3}/2)\,\cong\, 15.60340835$.
\item The Riemannian holonomy of $(V^\ir_{2,4}, g_{\alpha\beta\gamma})$
is  $\SO(5)$, i.\,e.~maximal.
\end{enumerate}
\end{thm}
\begin{proof}
The Levi-Civita connection  $\nabla^g$ is the unique metric connection whose
torsion vanishes, hence a calculation similar to that in the proof of
Theorem \ref{thm-connection} yields that $\nabla^g$ corresponds to the map
$\Lambda^g_\m: \, \m\ra\so(5)$,
\begin{gather*}
\Lambda^g_\m(e_1)\,=\, \left[\frac{\sqrt{\alpha}}{5\beta} - 
\frac{2}{\sqrt{\alpha}}\right]\, E_{23}+ \left[\frac{1}{\sqrt{\alpha}} - 
\frac{\sqrt{\alpha}}{10\gamma}\right]\, E_{45},\quad
\Lambda^g_\m(e_2)\,=\,\frac{\sqrt{\alpha}}{5\beta}E_{13},\\
\Lambda^g_\m(e_3)\,=\, - \frac{\sqrt{\alpha}}{5\beta}E_{12},\quad
\Lambda^g_\m(e_4)\,=\, - \frac{\sqrt{\alpha}}{10\gamma}E_{15},\quad
\Lambda^g_\m(e_5)\,=\,  \frac{\sqrt{\alpha}}{10\gamma}E_{14}.
\end{gather*}
The Riemannian curvature, its Ricci tensor and the Riemannian holonomy
then follow from a lengthy, but routine calculation, see the proof
of Theorem \ref{thm-connection} for the general method.

Let us investigate the Einstein condition $g=\kappa\cdot\Ric$. It leads to 
the equations
\bdm
\kappa\,=\, \frac{1}{\gamma} - \frac{1}{50\gamma^2},\quad
\kappa\,=\, \frac{2}{25\beta^2}+\frac{1}{50\gamma^2},\quad
\kappa\,=\, \frac{1}{\beta} - \frac{2}{25\beta^2},
\edm
which in turn are equivalent to
\bdm
25\,\kappa\beta^2-25\,\beta+2\,=\,0,\quad
50\,\kappa\gamma^2-50\,\gamma+1\,=\,0,\quad
\frac{2}{25\,\beta^2}+\frac{1}{50\gamma^2}-\kappa\,=\,0.
\edm
The last condition implies that $\kappa\neq 0$, the metric
cannot be Ricci-flat. In this case, the general solution of
the first two quadratic equations is
\bdm
\beta_{\pm}\ =\ \frac{5\pm\sqrt{25-8\kappa}}{10\kappa},\quad
\gamma_{\pm}\ =\ \frac{5\pm \sqrt{25-2\kappa}}{10\kappa}.
\edm
Inserting all four combinations into the last equation, we obtain
four possible conditions for an Einstein metric. One then checks that
the combination $(\beta_+, \gamma_+)$ yields the Einstein metric given
in the theorem (satisfying, in particular, $\kappa<25/8$), while the 
other combinations have no admissible solutions.
\end{proof}
\begin{thm}[Contact properties]\label{thm.contact}
%
\begin{enumerate}
\item[] 
\item $(V^\ir_{2,4}, g^{\alpha\beta\gamma})$ carries two invariant
normal almost contact  metric structures, characterized by
\bdm
\xi\  \cong\ \eta=\ e_1, \quad \vphi_{\pm}\ =\ -E_{23} \pm E_{45},\quad
dF_{\pm}\ =\ 0.
\edm
Both normal almost contact metric structures 
admit a unique characteristic connection with the same
characteristic torsion
\bdm
T^c\ =\ \eta\wedge d\eta\ =\  \frac{2\sqrt{\alpha}}{5\beta}\,e_1\wedge e_2\wedge
e_3 - \frac{\sqrt{\alpha}}{5\gamma}\, e_1\wedge e_4\wedge e_5.
\edm
For a metric $g_{\alpha\beta\gamma}$ defining a nearly integrable $\SO(3)_\ir$
structure, this connection coincides with the characteristic connection of
the $\SO(3)_\ir$ structure.
\item The  invariant
normal almost contact  metric structure $(\xi,\eta,\vphi_{+})$ is
Sasakian if and only if $\alpha=25\beta^2=100\gamma^2$; it is in
addition an $\SO(3)_\ir$ structure for 
$(\alpha,\beta,\gamma)=(\frac{25}{36}, \frac{1}{6},\frac{1}{12})$. 

\end{enumerate}
\end{thm}
\begin{proof}
An almost contact structure ist given by a vector field $\xi$ with dual
$1$-form $\eta$ and a $(1,1)$-tensor field $\vphi$ such that $\vphi^2=-\Id+
\eta\ox\xi$. In order to become an almost contact \emph{metric} structure,
a compatibility condition about the underlying Riemannian metric is required,
\bdm
g(\vphi(X),\vphi(Y))\ =\ g(X,Y)-\eta(X)\eta(Y).
\edm
With the Sasaki case already in mind, $e_1$ -- being a Killing vector field --
is a natural choice for $\xi \cong \eta$. Furthermore, it is reasonable
in our setting to restrict our attention to invariant structures, i.\,e.~endomorphisms $\vphi:\ \langle e_1 \rangle^\perp \ra \langle e_1 \rangle^\perp$ that 
are invariant under the isotropy representation. These are precisely
the endomorphisms commuting with $d\lambda(\tilde{e}_0)$, hence
$\vphi$ has to be of the form $\eps_1 E_{23} + \eps_2 E_{45},\ \eps_{1,2}=\pm 1$.
This leads to the two only inequivalent cases described in (1). The almost 
contact metric structure will be called \emph{normal} if its
Nijenhuis tensor $N_\vphi$ vanishes, 
\bdm
N_{\vphi}(X,Y)\ =\ [\vphi(X),\vphi(Y)] - \vphi [X,\vphi(Y)] - 
 \vphi [\vphi(X),Y] +\vphi^2[X,Y]+ d\eta(X,Y)\xi.
\edm
One computes $d\eta=\frac{2\sqrt{\alpha}}{5\beta} e_2\wedge e_3- 
\frac{\sqrt{5}}{5\gamma}e_4\wedge e_5$ and uses this fact to check
that, indeed, $N_{\vphi_\pm}=0$. The fundamental form of a contact
structure is defined by $F(X,Y):=g(X,\vphi(Y))$, hence one obtains
in our case
\bdm
F_{\pm}\ = \ e_2\wedge e_3 \mp e_4\wedge e_5.
\edm
By a routine calculcation, onw shows that
$d(e_2\wedge e_3)=d (e_4\wedge e_5)=0$, hence $dF_{\pm}=0$ and the general
expression for the torsion of an almost contact metric structure
 \cite[Thm 8.2]{Friedrich&I1} is reduced to $T^c=\eta\wedge d\eta$, and
hence yields the formula  stated in (1).

In order to become Sasakian, only $2F=d\eta$ still needs to be satisfied.
Thus, $F_{-}$ can never be Sasakian and $F_{+}$ has to satisfy the
condition $\alpha=25\beta^2=100\gamma^2$. This finishes the proof.
\end{proof}
\begin{NB}
Together with Theorem \ref{riem-curv-prop}, we conclude that all
Sasaki metrics on  $V^\ir_{2,4}$ are non-Einstein, contrary to the
standard Sasaki metric on $V^\st_{2,4}$ (see, for example, \cite[Ch. 4.3]{BFGK}
for a detailed discussion of this Einstein-Sasaki structure and its
properties). 

Geometrically speaking, the set of Sasaki metrics defines a roughly square root
shaped curve in the positive quadrant 
$\{(\alpha,\beta,\gamma)\in\R^3\ :\ \alpha>0,\ \beta>0,\ \gamma>0\}$
that intersects the ruled surface of nearly integrable $SO(3)_\ir$ metrics 
in exactly one point, as described in (2).
\end{NB}
The twisted Stiefel manifold $V^\ir_{2,4}=\SO(3)\x \SO(3)/ \SO(2)_\ir$
has a non-compact partner, namely, $\tilde{V}^\ir_{2,4}:=
\SO(2,1)\x \SO(3)/ \SO(2)_\ir$. We realize $\so(2,1)$ as
\bdm
\so(2,1)\ =\ \langle a_1, a_2, a_3 \rangle \text{ with basis }
a_1\, =\, \begin{bmatrix}0 & 0 & 0\\ 0 & 0 & -1\\ 0 & -1 & 0\end{bmatrix},\
a_2\, =\, \begin{bmatrix} 0 & 0 & 1 \\ 0 & 0 & 0\\ 1 & 0 & 0 \end{bmatrix},\
a_3\, =\, \begin{bmatrix} 0 & -1 & 0\\ 1 & 0 & 0 \\ 0 & 0 & 0\end{bmatrix}.
\edm
The commutator relations are $[a_1,a_2]=-a_3,\ [a_2,a_3]=a_1,\ [a_3,a_1]=a_2$.
For $\so(3)$, we use the same basis $b_1,\ b_2, b_3$ as before.
The embedding of $H=SO(2)$ into 
$G=\SO(2,1)\x \SO(3)$ is also unchanged, namely, with generator $a_3+2 b_3$,
and a good choice for a reductive complement $\m$ is
\bdm
\m\,=\, \n\oplus \m_1\oplus\m_2,\quad
\n\,=\, \R \cdot (b_3-2a_3),\quad
\m_1\,=\, \langle a_1,  a_2\rangle,\quad
\m_2\,=\, \langle  b_1,  b_2\rangle.
\edm
The isotropy representation is given by $\lambda:\so(2)\ra\so(5),\
\lambda (a_3+2b_3)=E_{23}+2 E_{45}$, i.\,e.~it turns out to agree with the result
in the compact case. A metric $g_{\alpha\beta\gamma}$ is then defined
by deformation in the three summands of $\m$ as before.

Yet, there are a few differences worth
noticing. We summarize the results in the following theorem without
proof (see \cite{JBB} for details).
\begin{thm}[Properties of $\tilde{V}^\ir_{2,4}=\SO(2,1)\x \SO(3)/ \SO(2)_\ir$]
The reductive homogeneous space $(\tilde{V}^\ir_{2,4}, g_{\alpha\beta\gamma})$ 
has the following properties:
\begin{enumerate}
\item For parameters $\alpha, \beta,\gamma >0$ satisfying
$-\alpha\beta+4\,\gamma\alpha+25\,\beta\gamma=0$, the $\SO(3)_\ir$ structure
is nearly integrable and the  torsion $T^{\alpha\beta\gamma}$ of its
characteristic connection $\nabla^{\alpha\beta\gamma}$  is, in a suitable
orthonormal basis, given by
\bdm
T^{\alpha\beta\gamma} \ =\ - \frac{2\sqrt{\alpha}}{5\beta}\,e_1\wedge e_2\wedge
e_3 - \frac{\sqrt{\alpha}}{5\gamma}\, e_1\wedge e_4\wedge e_5.
\edm
Its holonomy is $\SO(2)_\ir\subset \SO(5)$ and its torsion is parallel,
$\nabla^{\alpha\beta\gamma} T^{\alpha\beta\gamma} =0$.  
\item The metric of the nearly integrable $\SO(3)_\ir$ structure is never
naturally reductive and never Einstein.
\end{enumerate}
\end{thm}
\subsection{The manifold $W^\ir=\R\x (\SL(2,\R)\ltimes\R^2)/ \SO(2)_\ir$}
%
In contrast to the compact and non-compact twisted Stiefel manifold,
this space will turn out to carry a full $\SO(3)_\ir$ structure
(i.\,e.~not only a   $\SO(2)_\ir$ structure) and will prove that
the torsion of nearly integrable $\SO(3)_\ir$ structure is not necessarily
parallel. Let $G=\R\x (\SL(2,\R)\ltimes\R^2)$, and
take as standard basis of $\slin(2,\R)$ 
\bdm
X\ =\ \begin{bmatrix}1 & 0 \\ 0 & -1\end{bmatrix},\quad
E_+\ =\ \begin{bmatrix}0 & 1 \\ 0 & 0\end{bmatrix},\quad
E_-\ =\ \begin{bmatrix}0 & 0 & \\ 1 & 0\end{bmatrix},\quad
[X,E_\pm]\ =\ \pm 2E_\pm,\quad [E_+,E_-]\ =\ X.
\edm
We choose a basis for $\g=\R\oplus \slin(2,\R)\oplus \R^2$ that
depends on a parameter $\mu\in\R$,
\bdm
\bar e^\mu_0=E_+-E_-+\mu, \
\bar e^\mu_1= 1-\mu(E_+-E_-),\
\bar e_2=(0,1)^t,\  \bar e_3 = (1,0)^t,\
\bar e_4=E_+ + E_-,\
\bar e_5=X.
\edm
The element $\bar e^\mu_0$ generates a one-dimensional subgroup $H_\mu$ of $G$
isomorphic to $\SO(2)$, the isotropy representation is
$\lambda: \h_\mu \ra\so(5),\ \lambda(\bar e_0)=E_{23}+2 E_{45}$.
In particular, $\mu=0$ corresponds to the standard embedding
$\so(2)\ra \slin(2,\R)$; we will see later that, as to be expected,
$\mu=0$ is not admissible for a nearly integrable $\SO(3)_\ir$ structure.
The subspace 
\bdm
\m\ =\ \n^\mu \oplus \m_1\oplus \m_2,\quad
\n^\mu \ =\ \R\cdot \bar e^\mu_1,\quad 
\m_1\ =\  \langle \bar e_2, \bar e_3\rangle,\quad
\m_2\ =\  \langle \bar e_4, \bar e_5\rangle, 
\edm
is a reductive
complement of $\h_\mu$ in $\g$, and each if its summands is isotropy invariant.
We can thus define a homogeneous metric on $G/H_\mu$ by 
$g_{\alpha\beta\gamma}\ =\ \diag (\alpha, \beta,\beta,\gamma,\gamma),\ 
\alpha,\beta,\gamma\, >\, 0$. We drop the bar to denote the rescaled 
elements that form an orthonormal basis for this metric, and agree to write
$e^\mu_0$  for $\bar e^\mu_0$ as well. All non-vanishing commutator relations
between these new base vectors are
\begin{gather*}
[e_1^\mu,e_2]\ =\ - \frac{\mu}{\sqrt{\alpha}}e_3,\quad
[e_1^\mu,e_3]\ =\  \frac{\mu}{\sqrt{\alpha}}e_2,\quad
[e_1^\mu,e_4]\ =\ - \frac{2\mu}{\sqrt{\alpha}}e_5,\quad
[e_1^\mu,e_5]\ =\  \frac{2\mu}{\sqrt{\alpha}}e_4 \\
[e_2,e_4]\ =\ -\frac{1}{\sqrt{\gamma}}e_3,\quad
[e_2,e_5]\ =\ \frac{1}{\sqrt{\gamma}}e_2,\quad
[e_3,e_4]\ =\ -\frac{1}{\sqrt{\gamma}}e_2,\quad
[e_3,e_5]\ =\ -\frac{1}{\sqrt{\gamma}}e_3\\
[e_4,e_5]\ =\ \frac{2}{\gamma(\mu^2+1)}(\mu\sqrt{\alpha} e_1- e_0).
\end{gather*}
Observe that these do not only depend on the metric, but also on the embedding
parameter $\mu$.
\begin{thm}[Properties of $W^\ir=\R\x (\SL(2,\R)\ltimes\R^2)/ \SO(2)_\ir$]
The reductive homogeneous space $(W^\ir, g_{\alpha\beta\gamma})$ 
has the following properties:
\begin{enumerate}
\item For any $\beta>0$ and parameters $\alpha, \gamma >0$ satisfying
$\alpha\geq 12\gamma $, the 
$\SO(3)_\ir$ structure
is nearly integrable for the two embeddings of $\SO(2)\cong H_\mu\ra\SO(5)$  
defined by
\bdm
\mu\  =\  \frac{\sqrt{\alpha}\pm
\sqrt{\alpha-12\gamma}}{ 2\sqrt{3\gamma}}
\edm
 and the  torsion $T^{\alpha\beta\gamma}$ of its
characteristic connection $\nabla^{\alpha\beta\gamma}$  is, in a suitable
orthonormal basis, given by
\bdm
T^{\alpha\beta\gamma} \ =\  - \frac{2\sqrt{3}}{\sqrt{\gamma}}
\, (e_1\wedge e_2\wedge e_3 +  e_1\wedge e_4\wedge e_5).
\edm
Its holonomy is $\SO(3)_\ir\subset \SO(5)$. Its torsion is \emph{not}
parallel,
$\nabla^{\alpha\beta\gamma} T^{\alpha\beta\gamma} \neq 0$, but it is
divergence-free, $\delta T^{\alpha\beta\gamma}=0$.  
\item The metric of the nearly integrable $\SO(3)_\ir$ structure is never
naturally reductive and never Einstein.
\end{enumerate}
\end{thm}
\begin{proof}
In contrast to the spaces treated before, we choose the more appropriate
basis of $\so(3)_\ir\subset\so(5)$
\bdm
Y_3\ :=\ h\ =\ E_{23}+2E_{45},\quad
Y_1\ :=\ -\sqrt{3}E_{12}+E_{35}+E_{24},\quad
Y_2\ :=\ -\sqrt{3}E_{13}+E_{25}-E_{34}.
\edm
These elements satisfy the same commutator rules as the previously used 
basis $X_1,X_2,X_3$. We describe a metric connection 
$\Lambda_\m:\h_\mu\ra \so(3)_\ir$ and its invariant torsion $T$ via the 
standard  Ansatz ($a,b,c,m,n\in\R$)
\begin{gather*}
\Lambda_\m(e_1^\mu)\ =\ aY_3,\quad
\Lambda_\m(e_2^\mu)\ =\ b Y_1-cY_2,\quad 
\Lambda_\m(e_3)\ =\ c Y_1+b Y_2,\quad
\Lambda_\m(e_4)\ =\ \Lambda_\m(e_5)\ =\ 0,\\
T\ =\ m\, e_1^\mu\wedge e^\mu_2\wedge e_3 + n \, e^\mu_1\wedge e_4\wedge e_5.
\end{gather*}
Using the relation ($*$) between $T$ and $\Lambda_\m$, one checks that 
the coefficients have to satisfy
$b=0,\, c=1/\sqrt{\gamma}$ and
\bdm
m\ =\ a -\frac{\sqrt{3}}{\sqrt{\gamma}}+\frac{\mu}{\sqrt{\alpha}},\quad
n\ =\ 2a + 2 \frac{\mu}{\sqrt{\alpha}},\quad
m\ =\ - \frac{2\sqrt{3}}{\sqrt{\gamma}},\quad
n\ =\ -\frac{2\mu\sqrt{\alpha}}{\gamma(\mu^2+1)}.
\edm
Thus, a short calculation yields $a=-\sqrt{3/\gamma}-\mu/\sqrt{\alpha}$
and the quadratic equation for the embedding parameter $\mu$
\bdm
\sqrt{3}\gamma\,\mu^2-\sqrt{\alpha\gamma}\,\mu+\gamma\sqrt{3}\ =\ 0.
\edm
It has real solutions if and only if $\alpha\geq 12\gamma$;  for further use, 
we give the final expression for the non-vanishing parts of the connection 
map $\Lambda_\m$:
\bdm
\Lambda_\m(e_1)\ =\ - \left[\frac{\sqrt{3}}{\sqrt{\gamma}}
+\frac{\mu}{\sqrt{\alpha}}\right]Y_3,\quad
\Lambda_\m(e_2)\ =\ -\frac{1}{\sqrt{\gamma}}Y_2,\quad
\Lambda_\m(e_3)\ =\ \frac{1}{\sqrt{\gamma}}Y_2,
\edm
where it is understood that $\mu$ takes one of the two admissible
values. We  observe that $\Lambda_\m =0$ is not in the
the admissible parameter range, hence the nearly integrable $\SO(3)_\ir$
structure is never naturally reductive. Using formula $(**)$ for the
action of $\nabla^{\alpha\beta\gamma}$ on $2$-forms, one checks that $T$
has the non vanishing covariant derivatives
\bdm
\nabla^{\alpha\beta\gamma}_{e_2}T\ =\ \Lambda_\m(e_2)T\ =\
-\frac{6}{\gamma}e_3\wedge e_4\wedge e_5,\quad
\nabla^{\alpha\beta\gamma}_{e_3}T\ =\ \Lambda_\m(e_3)T\ =\
\frac{6}{\gamma}e_2\wedge e_4\wedge e_5 .
\edm
One next computes the map $\Lambda_\m^g$ representing the Levi-Civita
connection, 
\begin{gather*}
\Lambda^g_\m(e_1)\ =\ -\frac{\mu}{\sqrt{\alpha}}E_{23}-
\left[ \frac{2\mu}{\sqrt{\alpha}} +\frac{\sqrt{3}}{\sqrt{\gamma}} \right] E_{45},\quad
\Lambda^g_\m(e_2)\ =\ \frac{1}{\gamma}(E_{34}-E_{25}),\\
\Lambda^g_\m(e_3)\ =\ \frac{1}{\gamma}(E_{24}+E_{35}),\quad
\Lambda^g_\m(e_4)\ =\ -\frac{\sqrt{3}}{\sqrt{\gamma}}E_{15},\quad
\Lambda^g_\m(e_5)\ =\ \frac{\sqrt{3}}{\sqrt{\gamma}}E_{14}.
\end{gather*}
One deduces that the only non-vanishing  $\nabla^g$-derivatives of the 
elementary invariant forms 
$e_{123}$ and $e_{145}$ are
\bdm
\nabla^g_{e_2}e_{123}\ =\ \frac{1}{\gamma}(e_{135}+e_{124}),\ \
\nabla^g_{e_3}e_{123}\ =\ \frac{1}{\gamma}(e_{125} - e_{134}),\ \ 
\nabla^g_{e_4}e_{123}\ =\ - \frac{\sqrt{3}}{\sqrt{\gamma}}e_{235},\ \ 
\nabla^g_{e_4}e_{123}\ =\  \frac{\sqrt{3}}{\sqrt{\gamma}}e_{234}
\edm
as well as
\bdm
\nabla^g_{e_2}e_{145}\ =\ - \frac{1}{\gamma}(e_{135}+e_{124}),\quad
\nabla^g_{e_3}e_{145}\ =\  \frac{1}{\gamma}(e_{134} - e_{125}),\quad
\edm
One thus checks that 
\bdm
d e_{123}\ =\ -\frac{2\sqrt{3}}{\sqrt{\gamma}} e_{2345}, \quad
d e_{145}\ =\ 0,\quad dT\ =\  \frac{6}{\gamma} e_{2345} \neq\ 0,\quad
\delta T = -\sum_{i=1}^5 e_i\haken \nabla^g_{e_i}T = 0.
\edm
A computation of the Riemannian curvature tensor and its contraction shows that
the Ricci tensor is diagonal, with
\bdm
\Ric^g (e_1,e_1)\ =\ \frac{2\mu^2\alpha}{\gamma^2(\mu^2+1)},\quad
\Ric^g (e_2,e_2)\ =\ \Ric^g (e_3,e_3)\ =\ 0,\quad
\Ric^g (e_4,e_4)\ =\ \Ric^g (e_5,e_5),
\edm
and with the somehow uninspiring expression
\bdm
\Ric^g (e_4,e_4)\ =\ -\frac{(\alpha+6\gamma)( \mu^4 + 2\mu^2) +6\gamma }{\gamma^2(\mu^2+1)^2} .
\edm
However, it is plain that the space will never be Einstein.
\end{proof}
\begin{NB}
It is a natural question to ask in this case, as before for the
twisted Stiefel manifold, about compatible contact structures.
It turns out that the picture is rather different. As explained
in Theorem \ref{thm.contact} and its proof,
\bdm
\xi\  \cong\ \eta=\ e_1, \quad \vphi_{\pm}\ =\ -E_{23} \pm E_{45},
\edm
is a natural choice for  an almost contact structure (basically, because the
isotropy representation is the same). However,
one checks that the Nijenhuis tensor of $\vphi_+$ is not a $3$-form,
hence, by \cite[Thm 8.2]{Friedrich&I1}, it does not admit
an invariant metric connection with skew-symmetric torsion.
For $\vphi_-$,  the Nijenhuis tensor vanishes and $d F_-=0$,
but the corresponding contact connection has torsion
\bdm
T^-\ =\ \eta\wedge d\eta\ =\ -\frac{2\sqrt{3}}{\sqrt{\gamma}}\, e_{145}
\ \neq\ T^{\alpha\beta\gamma}.
\edm
Thus, the almost metric contact structure defined by $(\eta,\vphi_-)$
is not compatible with the $\SO(3)_\ir$ structure.
\end{NB}
\begin{NB}
In \cite[p.77 and p.88]{Bobienski&N07}, it was claimed that $W^\ir$ with the
standard embedding of $\SO(2)\ra \SL(2,\R)$ (corresponding to $\mu=0$
in our notation) carries an $\SO(3)_\ir$ structure (that was not further investigated). The results above
correct this point.
\end{NB}
%
    
\end{document}